\documentclass[a4paper,10pt]{article}
\usepackage[hmarginratio=1:1, bottom=3.5cm, top=2.5cm]{geometry}
\usepackage{amssymb,amsthm,amsmath,latexsym,url,graphicx, comment}
\usepackage[english]{babel} 
\usepackage[dvipsnames]{xcolor}
\usepackage[utf8]{inputenc}
\usepackage[sc,osf]{mathpazo}
\usepackage[euler-digits]{eulervm}
\usepackage[T1]{fontenc}
\usepackage{mathtools}
 \usepackage{enumitem} 
\usepackage{pdfpages}
\usepackage{float}
\usepackage[colorlinks=true,pdftex,unicode=true,linktocpage,bookmarksopen,hypertexnames=false]{hyperref}
\theoremstyle{plain}

\newtheorem{theorem}{Theorem}[section]
\newtheorem{lemma}[theorem]{Lemma}
\newtheorem{proposition}[theorem]{Proposition}
\newtheorem{corollary}[theorem]{Corollary}
\theoremstyle{remark}
\newtheorem{remark}[theorem]{Remark}

\newcommand{\F}{\mathbb F}
\newcommand{\Z}{\mathbb Z}
\newcommand{\Aut}{\operatorname{Aut}}
\newcommand{\Core}{\operatorname{Core}}
\newcommand{\Fix}{\operatorname{Fix}}
\newcommand{\Stab}{\operatorname{Stab}}
\newcommand{\FC}{\operatorname{FC}}

\newcommand{\Triv}{\operatorname{Triv}}
\newcommand{\FSB}{\operatorname{FSB}}

\title{Finite-index problems in skew braces}
\author{Massimiliano Di Matteo \and Maria Ferrara}
\date{ }

\begin{document}
\maketitle

\begin{abstract} 
We investigate finite-index problems in skew braces. For every sub-skew brace \(A\) of a skew brace \(B\), we prove that finite additive index is equivalent to finite multiplicative index; whenever these indices are finite, they coincide. This answers Question~3.7 of \cite{CPV} affirmatively. We then construct a left brace with a strong left ideal of index \(3\) containing no finite-index ideal, giving a negative answer to Question~3.6 of \cite{CPV}. We also show that finite additive and multiplicative conjugacy classes, together with a finite \(\lambda\)-orbit, force an element to be an \((s)\)-element, thereby answering Question~5.21 of \cite{CPV}. Finally, for every \(n\geq 3\), a one-generated free right-nilpotent skew brace of class \(n\), introduced in \cite{Free}, has an index-\(2\) ideal that is not finitely generated as a skew brace, although it is finitely generated as an ideal.\\


\end{abstract}

\vspace{1cm}
\section{Introduction}
Finite-index conditions play a fundamental role in algebra. In group theory, for instance, they control normal cores and finite generation. For skew braces, however, every such question involves two interacting group structures, and familiar group-theoretic principles split into genuinely different problems.

The purpose of this paper is to answer some finite-index questions for skew braces. Our first result compares the two group indices. For an arbitrary sub-skew brace \(A\) of \(B\), we prove
\[
   |B:A|_+<\infty \quad\Longleftrightarrow\quad |B:A|_\circ<\infty,
\]
and the two indices are equal whenever they are finite. This completely answers \cite[Question~3.7]{CPV}. This extends the index-property previously established for locally almost polycyclic
skew braces \cite{poly}, including locally supersoluble skew braces, in \cite[Corollary 4.17]{supersoluble}.

The second problem concerns ideal cores. It was known that every finite-index sub-skew brace contains a finite-index strong left ideal, and an ideal can be obtained under additional hypotheses \cite[Corollary~3.5]{CPV}. We show that these hypotheses cannot be removed: there exists a left brace \(B\) with a strong left ideal \(A\) of index \(3\) such that every ideal contained in \(A\) has infinite index. This answers Question~3.6 of \cite{CPV} negatively. . For comparison, every index-2 sub-skew brace of a locally supersoluble skew brace is an ideal \cite[Theorem 4.18]{supersoluble}.

The two problems above had already been solved simultaneously for almost polycyclic
skew braces: every finite-index sub-skew brace contains a finite-index ideal, and finite addi-
tive index is equivalent to finite multiplicative index, with the same \cite[Corollaries 3.11 and 3.12]{poly}. Thus our first theorem removes the almost-polycyclic hypothesis from the comparison of indices, whereas our counterexample shows that no analogous unrestricted ideal-core
theorem can hold.

Our third contribution concerns finite conjugacy and fixed points. If the additive and multiplicative conjugacy classes of an element \(x\), together with its \(\lambda\)-orbit, are finite, we prove that \(\operatorname{Fix}(\lambda_x)\) contains a finite-index left ideal. It follows that \(x\) is an \((s)\)-element, yielding an affirmative answer to Question~5.21 of \cite{CPV}.

Finally, we investigate the analogue of Schreier's lemma \cite{Schreier}. Although finite-index sub-skew braces are finitely generated whenever one underlying group of the ambient brace is finitely generated, the conclusion fails under finite generation merely as a skew brace. For every \(n\geq 3\), a one-generated free right-nilpotent skew brace of class \(n\) contains an index-\(2\) ideal that is not finitely generated as a skew brace.

Section~2 fixes notation and terminology. Section~3 proves the index-comparison theorem, Section~4 constructs the ideal-core counterexample, Section~5 settles the fixed-point problem, and Section~6 contains the positive and negative finite-generation results.

\section{Preliminaries}
Let $B=(B,+,\circ)$ be a (left) skew brace \cite{GV}. Thus $(B, +)$ and $(B, \circ)$ are groups and $$a \circ (b+c) = a \circ b-a+a \circ c,$$ for all $a, b, c \in B$.
If the additive group $(B, +)$ is abelian, then $B$ is called a \textit{(left) brace}. Historically, the latter form a really important class of skew braces, as they were the first to be defined in \cite{Rump}, and paved the way for the study of skew brace and their relationship with the Yang--Baxter Equation.  
We write $\bar{a}$ for the inverse of $a$ in $(B,\circ)$ and reserve $-a$ for its inverse in $(B,+)$. The maps
\[
\lambda_a(b)=-a+a\circ b
\]
define a homomorphism
\[
\lambda\colon (B,\circ)\longrightarrow \operatorname{Aut}(B,+).
\]
We also set
\[
a*b=\lambda_a(b)-b.
\]

The most general substructure of a skew brace is the \textit{sub-skew brace}, which is just a subgroup of both the additive and the multiplicative group. A {\it left ideal} is a subgroup of $(B,+)$ invariant under the map $\lambda_a$ for all $a$ in $B$. It is called a \emph{strong left ideal} if it is normal in $(B,+)$, and an \emph{ideal} if it is, in addition, normal in $(B,\circ)$. Ideals are exactly the class of substructures that induce a well posed skew brace structure for the quotient set $B/I$.

The notion of index is from \cite[Definition~4.22]{BBEFPT}; in particular, given a sub-skew brace $S$ of $B$, if $|(B,+):(S,+)|=|(B,\circ):(S,\circ)|$, then we can define the {\it index} of $S$ in $B$ as the cardinal number $|B:S|:=|(B,+):(S,+)|=|(B,\circ):(S,\circ)|$.
It is not clear if one can actually define the index for any sub-skew brace. The restriction of this problem to the finite case is called \textit{the finite-index problem} \cite[Question~3.7]{CPV} and can be rephrased as follows
\begin{center}
    \textit{Let $(B,+,\circ)$ a skew brace and let $S$ be a sub-skew brace of $B$, then $|(B,+):(S,+)|<\infty$ if and only if $|(B,\circ):(S,\circ)|<\infty$.}
\end{center}
The index is well-defined for serial sub-skew braces and hence for finite-index sub-skew braces of locally centrally nilpotent left skew braces \cite[Lemma~4.24 and Theorem~4.25]{BBEFPT}. Moreover, the two groups indices are known to be equal whenever $S$ is a left ideal \cite[Remark~3.1]{CPV} or whenever both are finite \cite[Theorem~3.3]{CPV}. In section 3, we prove that this is always the case when just one of the two indices is finite, as a positive answer to the finite-index problem is given. 

In \cite{BBEFPT}, the finite-index core problem is addressed by proving that the maximal ideal of $B$ contained in $C$ has finite index whenever $B/\zeta_2(B)$ is finite \cite[Theorem~4.26]{BBEFPT}. Furthermore, it was proved in \cite[Proposition~3.2]{CPV} that every finite-index
sub-skew brace contains a finite-index strong left ideal.  For two-sided skew braces they obtained an
ideal \cite[Corollary~3.5]{CPV}, and asked whether this holds in general
\cite[Question~3.6]{CPV}.  Di Matteo, Esteban--Romero, Ferrara and
P\'erez--Calabuig posed the equivalent core problem
\cite[Question~3.9]{DMEFP} and proved it under finite-generation and
finite-index assumptions on the upper socle series or the hypercentre
\cite[Proposition~3.10 and Corollary~3.11]{DMEFP}. 
The following theorem gives
a negative answer, already among left braces.

\begin{theorem}\label{thm:no-ideal-core}
There exist a left brace $B$ and a strong left ideal $A$ of $B$ such that
\[
  |B:A|=3,
\]
but every ideal $I$ of $B$ contained in $A$ has infinite index in $B$.
\end{theorem}
The notion of property~$(S)$ and its finite-conjugacy framework were
introduced by Colazzo, Ferrara and Trombetti in \cite{CFT};
Question~5.21 itself was subsequently formulated in \cite{CPV}.
In general, conjugacy-like arguments for skew braces must also take
into account the behavior of the lambda function. For $x\in B$, set
\begin{align*}
 [x]_+&=\{-a+x+a:a\in B\},\\
 [x]_\circ&=\{a\circ x\circ\bar a:a\in B\},\\
 [x]_\lambda&=\{\lambda_a(x):a\in B\}.
\end{align*}
The set of elements with finite $\lambda$-orbit is denoted
$\lambda_f(B)$, while $\FC(B,+)$ and $\FC(B,\circ)$ are the two
FC-centres.  We use
\[
 \theta_f(B)=\lambda_f(B)\cap\FC(B,+).
\]
Equivalently, $x\in\theta_f(B)$ precisely when its orbit
\[
 [x]_\theta=\{a+\lambda_b(x)-a:a,b\in B\}
\]
is finite.
An element $x$ is an $(s)$-element if each of the four sets
\[
 \{g*x:g\in B\},\quad \{x*g:g\in B\},\quad [x]_+,
 \quad [x]_\circ
\]
is finite.

If the additive and multiplicative conjugacy classes
and the $\lambda$-orbit of $x$ are finite, Question~5.21 of \cite{CPV} asks
whether the remaining right star orbit is finite.  We prove more: the
fixed-point subgroup of $\lambda_x$ contains a left ideal of finite
index, with an explicit bound.  Consequently, Question~5.21 has an
affirmative answer for arbitrary skew braces; see
Theorem~\ref{thm:s-element}.

\bigskip
The last question is a skew-brace version of the classical Schreier's lemma: a finite-index subgroup of a finitely generated group is finitely
generated.  It has a positive answer as soon as either underlying group
of the ambient brace is finitely generated.  It fails, however, when
one assumes only finite generation as a skew brace.  Using the free
right-nilpotent skew braces constructed by Jespers, Letourmy, Properzi,
Trombetti and Van Antwerpen \cite{Free}, we prove that for every
$n\geq3$ a one-generated skew brace $B_n$ has an ideal $T_n$ of index
$2$ which is not finitely generated as a skew brace; see
Theorem~\ref{thm:schreier-fails}.  By a lemma of Trombetti, $T_n$ is
still finitely generated as an \emph{ideal} of $B_n$.  This distinction
is essential.  In class $2$ the same parity ideal is instead generated
by four explicit elements.

For completeness, we also record a finite-presentability consequence of results of \cite{CPV} together with Trombetti's extension theorem \cite{Trombetti}. Namely, every finitely generated
$\theta_f$-skew brace is finitely presented and virtually trivial. This is included as an application of the cited results, rather than as an independent finite-presentation theorem.

\section{Index comparison}
Question~3.7 of \cite{CPV} asks whether, for an
arbitrary sub-skew brace \(A\) of a skew brace \(B\), finiteness of one of
\(\lvert B:A\rvert_{+}\) and \(\lvert B:A\rvert_{\circ}\) is equivalent to
finiteness of the other. her. A positive answer was
already known when $B$ is locally almost polycyclic \cite{poly}, and hence when it is locally supersoluble \cite[Corollary 4.17]{supersoluble}. We first recall these group-theoretic results.

\begin{lemma}\label{lem:neumann-covering}
If a group \(G\) is the union of \(m\) cosets of subgroups
\(G_1,\dots,G_m\), then at least one \(G_i\) has finite index in \(G\).
Moreover, one may choose \(i\) such that \(\lvert G:G_i\rvert\leq m\).
\end{lemma}
\begin{lemma}\label{lem:coset-packing}
Let \(G\) be a group, let \(n\) be a positive integer, and let
\(\{g_iH_i:i\in I\}\) be a family of pairwise disjoint left cosets in \(G\).
If \(\lvert G:H_i\rvert=n\) for every \(i\in I\), then \(\lvert I\rvert\leq n\).
\end{lemma}
Lemma~\ref{lem:neumann-covering} is Neumann's lemma on finite coset 
coverings \cite{Neumann}.

\begin{theorem}\label{thm:index-comparison}
Let \(A\) be a sub-skew brace of a skew brace \(B\). Then
\[
  \lvert B:A\rvert_{+}<\infty
  \quad\Longleftrightarrow\quad
  \lvert B:A\rvert_{\circ}<\infty.
\]
Whenever these conditions hold,
\[
  \lvert B:A\rvert_{+}=\lvert B:A\rvert_{\circ}.
\]
Consequently, Question~3.7 of \cite{CPV} has an
affirmative answer.
\end{theorem}

\begin{proof}
Assume first that \(\lvert B:A\rvert_{\circ}<\infty\), and let \(T\) be a
finite left transversal for \(A\) in \((B,\circ)\). For each \(t\in T\), the
brace identity gives
\[
  t\circ A=t+\lambda_t(A).
\]
Thus the additive group is covered by finitely many cosets of additive
subgroups:
\[
  (B,+)=\bigcup_{t\in T}\bigl(t+\lambda_t(A)\bigr).
\]
By \ref{lem:neumann-covering}, at least one subgroup \(\lambda_t(A)\) has
finite index in \((B,+)\). Since \(\lambda_t\) is an additive automorphism,
\(A\) itself has finite additive index.

Conversely, put \(n\coloneqq\lvert B:A\rvert_{+}<\infty\), and let \(T\) be an
arbitrary left transversal for \(A\) in \((B,\circ)\); at this point \(T\) is
not assumed finite. The multiplicative cosets \(t\circ A\), for \(t\in T\),
are pairwise disjoint. Viewed inside the additive group, they are
\[
  t\circ A=t+\lambda_t(A).
\]
Every \(\lambda_t(A)\) is an additive subgroup of index \(n\), because
\(\lambda_t\in\Aut(B,+)\). Applying Lemma \ref{lem:coset-packing} to \((B,+)\)
gives \(\lvert T\rvert\leq n\). Hence
\(\lvert B:A\rvert_{\circ}<\infty\).

In either direction both indices are finite, and their equality follows from
\cite[Theorem~3.3]{CPV}.
\end{proof}

\section{The finite-index ideal-core problem}\label{sec:ideal-core}

The maximal ideal contained in a sub-skew brace has finite index under the
central-nilpotency hypothesis $B/\zeta_2(B)$ finite
\cite[Theorem~4.26]{BBEFPT}.  Di Matteo, Esteban--Romero, Ferrara and
P\'erez--Calabuig posed the general core problem in
\cite[Question~3.9]{DMEFP} and obtained positive results under
finite-generation and finite-index hypotheses on the upper socle
series or the hypercentre
\cite[Proposition~3.10 and Corollary~3.11]{DMEFP}.  We now prove that
some additional hypothesis is necessary.

The construction of Theorem~\ref{thm:no-ideal-core} has two
ingredients.  First we produce a left brace $F$ of order $16$ with a
multiplicative epimorphism $\chi:(F,\circ)\to C_2$ whose kernel
contains only a very small left ideal.  We then take a restricted direct
sum of copies of $F$ and form a semidirect extension by the trivial
brace $C_3$.

\subsection{A left brace of order sixteen}

Let $V=\F_2^4$.  We write
\[
  x=(a,b,c,d), \qquad y=(u,v,w,z),
\]
and define
\begin{align*}
  p(x)&=b+c+d+cd,\\
  q(x)&=c+bd,\\
  r(x)&=b+c+cd.
\end{align*}
All coordinates and polynomial expressions in this section are over
$\F_2$.  For $x\in V$, define an additive automorphism $\lambda_x$ by
\begin{equation}\label{eq:lambda-matrix}
 \lambda_x=
 \begin{pmatrix}
  1&0&0&p(x)\\
  0&1&d&q(x)\\
  0&0&1&r(x)\\
  0&0&0&1
 \end{pmatrix}.
\end{equation}
Equivalently,
\begin{equation}\label{eq:lambda-coordinates}
 \lambda_x(y)=
 \bigl(
 u+z(b+c+d+cd),\,
 v+dw+z(c+bd),\,
 w+z(b+c+cd),\,
 z
 \bigr).
\end{equation}

\begin{lemma}\label{lem:finite-brace}
The operation
\[
  x\circ y=x+\lambda_x(y)
\]
makes $F=(V,+,\circ)$ into a left brace.
\end{lemma}

\begin{proof}
The matrices in \eqref{eq:lambda-matrix} are invertible.  If
$M(p,d,q,r)$ denotes the displayed matrix, then
\begin{equation}\label{eq:matrix-product}
 M(p,d,q,r)M(p',d',q',r')
 =M(p+p',d+d',q+q'+dr',r+r').
\end{equation}
Let $t=x+\lambda_x(y)$.  Substitution in the definitions of $p,q,r$
gives
\begin{align*}
 p(t)&=p(x)+p(y),& d(t)&=d+z,\\
 r(t)&=r(x)+r(y),&
 q(t)&=q(x)+q(y)+d\,r(y).
\end{align*}
Together with \eqref{eq:matrix-product}, these identities yield
\[
  \lambda_{x+\lambda_x(y)}=\lambda_x\lambda_y.
\]
Consequently the maps $x\mapsto (x,\lambda_x)$ form a regular subgroup
of the holomorph of $(V,+)$, and the stated operation defines a left
brace.
\end{proof}

Define
\begin{equation}\label{eq:chi}
 \chi:(F,\circ)\longrightarrow C_2,
 \qquad \chi(x)=p(x)=b+c+d+cd.
\end{equation}
The first identity in the proof of Lemma~\ref{lem:finite-brace} shows
that $\chi$ is a group homomorphism.  It is surjective, since
$\chi(0,1,0,0)=1$.

For the elementary calculations below, identify $(a,b,c,d)$ with the
integer $a+2b+4c+8d$ in $\{0,\ldots,15\}$.  This notation is used only
as a label; addition remains vector addition over $\F_2$.

\begin{lemma}\label{lem:kernel-left-ideals}
Let
\[
  J=\langle(1,0,0,0)\rangle_+=\{0,1\}.
\]
Then:
\begin{enumerate}[label=\textup{(\roman*)}]
 \item $J$ is an ideal of $F$;
 \item every left ideal of $F$ contained in $\ker\chi$ is contained in
       $J$;
 \item the multiplicative group of $F/J$ is nonabelian.
\end{enumerate}
\end{lemma}

\begin{proof}
The matrices $\lambda_x$ fix $(1,0,0,0)$ for every $x$, while
$\lambda_j=\mathrm{id}$ for every $j\in J$.  Thus $J$ is
$\lambda$-invariant.  Moreover, for $x\in F$ and $j\in J$,
\[
  x\circ j\circ x^{-1}_\circ=\lambda_x(j)\in J.
\]
Hence $J$ is also normal in $(F,\circ)$, and is therefore an ideal.

From \eqref{eq:chi},
\[
 \ker\chi=\{0,1,6,7,10,11,14,15\}.
\]
Formula \eqref{eq:lambda-coordinates} gives
\begin{align*}
 \lambda_8(6)&=4,& \lambda_8(7)&=5,\\
 \lambda_4(10)&=13,& \lambda_4(11)&=12,\\
 \lambda_4(14)&=9,& \lambda_4(15)&=8.
\end{align*}
Every element on the right lies outside $\ker\chi$.  Therefore a
$\lambda$-invariant additive subgroup contained in $\ker\chi$ cannot
contain any of $6,7,10,11,14,15$, and so it is contained in $J$.

Finally, a direct calculation gives
\[
  2\circ 8=15, \qquad 8\circ 2=10.
\]
The additive difference $15-10=5$ does not belong to $J$.  Thus the two
products define distinct cosets modulo $J$, proving that $(F/J,\circ)$
is nonabelian.
\end{proof}

\begin{remark}\label{rem:core-fails}
The finite brace $F$ already displays the obstruction to the usual core
argument.  The set
\[
  L=\{0,2,4,6\}
\]
is a strong left ideal of $F$, but a direct calculation gives
\[
  \Core_{(F,\circ)}(L)=\{0,6\}.
\]
This multiplicative core is not a left ideal, since
$\lambda_8(6)=4$.  Thus taking the multiplicative core of a strong left
ideal need not preserve $\lambda$-invariance.
\end{remark}

\subsection{An infinite obstruction}

Let
\[
  C=\bigoplus_{n\geq 1}F_n,
  \qquad F_n\cong F,
\]
be the restricted direct sum, with both brace operations defined
coordinatewise.  For $c=(c_n)\in C$, set
\begin{equation}\label{eq:chi-infinity}
  \chi_\infty(c)=\sum_{n\geq 1}\chi(c_n)\in C_2.
\end{equation}
The sum is finite, and $\chi_\infty:(C,\circ)\to C_2$ is a group
epimorphism.

\begin{lemma}\label{lem:no-finite-index-kernel}
No finite-index ideal of $C$ is contained in $\ker\chi_\infty$.
\end{lemma}

\begin{proof}
Suppose that $K$ is an ideal of $C$, that $K\subseteq\ker\chi_\infty$,
and that $D=C/K$ is finite.  For every $n$, the coordinate brace $F_n$
is an ideal of $C$, and hence
\[
  K_n=K\cap F_n
\]
is an ideal of $F_n$.  An element supported in the $n$th coordinate
belongs to $\ker\chi_\infty$ precisely when its nonzero coordinate
belongs to $\ker\chi$.  Lemma~\ref{lem:kernel-left-ideals} therefore
implies
\[
  K_n\subseteq J_n,
\]
where $J_n\cong J$ is the corresponding ideal of $F_n$.

Let $H_n$ be the image of $(F_n,\circ)$ in the finite group
$(D,\circ)$.  Then
\[
  H_n\cong (F_n/K_n,\circ).
\]
Since $K_n\subseteq J_n$ and $(F_n/J_n,\circ)$ is non-abelian, every
$H_n$ is non-abelian.  If $n\neq m$, then $F_n$ and $F_m$ commute
elementwise in $(C,\circ)$, so $H_n$ and $H_m$ commute elementwise in
$(D,\circ)$.

The subgroups $H_n$ are pairwise distinct.  Indeed, if $H_n=H_m$ for
some $n\neq m$, then any two elements of this common subgroup may be
lifted respectively to $F_n$ and $F_m$, and therefore commute.  This
would make $H_n$ abelian, a contradiction.  The finite group
$(D,\circ)$ would thus contain infinitely many distinct subgroups
$H_n$, which is impossible.
\end{proof}

\subsection{The semidirect extension}

Let $T=C_3$ with its trivial brace structure.  Define an action of
$(C,\circ)$ on $(T,+)$ by
\[
  \varepsilon_c(q)=(-1)^{\chi_\infty(c)}q.
\]
Since $\chi_\infty$ is a multiplicative homomorphism, so is
$c\mapsto\varepsilon_c$.  On the set
\[
  B=T\times C
\]
define
\begin{align}
 (q,c)+(s,e)&=(q+s,c+e),\label{eq:addition-B}\\
 (q,c)\circ(s,e)&=(q+\varepsilon_c(s),c\circ e).
 \label{eq:multiplication-B}
\end{align}

\begin{lemma}\label{lem:B-brace}
Equations \eqref{eq:addition-B} and \eqref{eq:multiplication-B} define
a left brace structure on $B$, with
\begin{equation}\label{eq:lambda-B}
  \lambda_{(q,c)}(s,e)
  =\bigl(\varepsilon_c(s),\lambda_c(e)\bigr).
\end{equation}
\end{lemma}

\begin{proof}
The multiplicative group is the semidirect product
$C_3\rtimes_\varepsilon(C,\circ)$.  Formula \eqref{eq:lambda-B}
follows immediately from the definitions.  The assignment
\[
  (q,c)\longmapsto(\varepsilon_c,\lambda_c)
\]
is a homomorphism from $(B,\circ)$ to
$\Aut(C_3\times(C,+))$, because both $c\mapsto\varepsilon_c$ and
$c\mapsto\lambda_c$ are multiplicative homomorphisms.  Hence the skew
brace identity holds.  Since the additive group is abelian, $B$ is a
left brace.
\end{proof}

\begin{proof}[Proof of Theorem~\ref{thm:no-ideal-core}]
Set
\[
  A=\{0\}\times C.
\]
It is an additive subgroup of index $3$.  Formula
\eqref{eq:lambda-B} shows that it is invariant under every lambda map.
As $(B,+)$ is abelian, $A$ is a strong left ideal.  In particular its
additive and multiplicative indices coincide, and $|B:A|=3$.

Suppose that $I$ is an ideal of $B$ contained in $A$.  There is an
additive subgroup $K$ of $C$ such that
\[
  I=\{0\}\times K.
\]
Restricting additive conjugation, multiplicative conjugation and the
lambda action to $\{0\}\times C$ shows that $K$ is an ideal of $C$.

We claim that $K\subseteq\ker\chi_\infty$.  For $q\in C_3$ and
$k\in K$, multiplicative normality of $I$ gives
\begin{align*}
 &(q,0)\circ(0,k)\circ(q,0)^{-1}_\circ\\
 &\hspace{2cm}=
 \bigl(q-\varepsilon_k(q),k\bigr)\in I.
\end{align*}
Here $(q,0)^{-1}_\circ=(-q,0)$.
If $\chi_\infty(k)=1$, then $\varepsilon_k(q)=-q$, and choosing
$q\neq0$ gives $q-\varepsilon_k(q)=2q\neq0$ in $C_3$.  This contradicts
$I\subseteq A$, proving the claim.

If $I$ had finite index in $B$, then $K$ would have finite index in
$C$.  This is impossible by Lemma~\ref{lem:no-finite-index-kernel}.
\end{proof}

\begin{corollary}
Question~3.6 of \cite{CPV} and Question~3.9 of \cite{DMEFP} have
negative answers, even within the class of left braces.  Thus the
finite-index core conclusion of \cite[Theorem~4.26]{BBEFPT} does not
extend to arbitrary left braces.
\end{corollary}

\section{Finite-index fixed points}\label{sec:fixed-points}
We turn to Question~5.21 of \cite{CPV}.  The main step is stronger
than the required finiteness of the right star orbit: a finite-index
left ideal is contained in the fixed-point subgroup of $\lambda_x$.

For an additive automorphism $\varphi$ of $(B,+)$, write
\[
  \Fix(\varphi)=\{b\in B:\varphi(b)=b\}.
\]

\begin{proposition}\label{prop:fixed-left-ideal}
Let $B$ be a skew brace and let $x\in B$.  Suppose that $[x]_+$,
$[x]_\circ$, and $[x]_\lambda$ are finite.  Then there is a left ideal
$M$ of $B$ such that
\[
  |B:M|_+<\infty
  \qquad\text{and}\qquad
  M\subseteq\Fix(\lambda_x).
\]
More precisely, if
\[
  a=|[x]_+|,\qquad c=|[x]_\circ|,
  \qquad \ell=|[x]_\lambda|,
\]
then $M$ may be chosen so that
\begin{equation}\label{eq:index-bound}
  |B:M|_+\leq (ac\ell)^c.
\end{equation}
Moreover, $\Fix(\lambda_x)$ contains a strong left ideal of finite
index in $B$.
\end{proposition}

\begin{proof}
Set
\[
 H=C_{(B,\circ)}(x)\cap\Stab_\lambda(x),
 \qquad
 \Stab_\lambda(x)=\{b\in B:\lambda_b(x)=x\}.
\]
This is a subgroup of $(B,\circ)$ and
\begin{equation}\label{eq:H-index}
  |B:H|_\circ\leq c\ell.
\end{equation}

For $b\in B$, define a permutation of $B$ by
\[
  \delta_b(z)=b+\lambda_b(z)-b.
\]
The diagonal map
\[
 (B,\circ)\longrightarrow (B,+)\rtimes_\lambda(B,\circ),
 \qquad b\longmapsto(b,b),
\]
is a group homomorphism, since
\[
  (b,b)(d,d)=(b+\lambda_b(d),b\circ d)=(b\circ d,b\circ d).
\]
Consequently, $b\mapsto\delta_b$ is an action of $(B,\circ)$ on $B$.
For $h\in H$ one has $\lambda_h(x)=x$, and therefore
\[
  \delta_h(x)=h+x-h.
\]
Let
\[
 A_x=\Stab_H(x)=H\cap C_{(B,+)}(x),
\]
where the stabilizer is taken for the $\delta$-action.  It follows that
$A_x$ is a subgroup of $(B,\circ)$ and
\begin{equation}\label{eq:Ax-index}
  |H:A_x|_\circ\leq a.
\end{equation}

Every element of $A_x$ is fixed by $\lambda_x$.  Indeed, for $u\in
A_x$, the definitions of $H$ and $A_x$ give
\[
 u\circ x=u+\lambda_u(x)=u+x=x+u
\]
and
\[
 u\circ x=x\circ u=x+\lambda_x(u).
\]
Left cancellation in $(B,+)$ yields $\lambda_x(u)=u$.  Hence
\begin{equation}\label{eq:Ax-in-fix}
  A_x\subseteq F_x:=\Fix(\lambda_x).
\end{equation}

Let $T$ be a left transversal for $A_x$ in $(B,\circ)$.  By
\eqref{eq:H-index} and \eqref{eq:Ax-index}, it is finite and
\begin{equation}\label{eq:T-bound}
 |T|=|B:A_x|_\circ\leq ac\ell.
\end{equation}
For $t\in T$, put
\[
  x_t=t\circ x\circ\bar t,
  \qquad F_t=\Fix(\lambda_{x_t}).
\]
If $u\in A_x$, then \eqref{eq:Ax-in-fix} and the homomorphism property
of $\lambda$ give
\[
 \lambda_{x_t}(\lambda_t(u))
 =\lambda_t\lambda_x\lambda_t^{-1}(\lambda_t(u))
 =\lambda_t\lambda_x(u)
 =\lambda_t(u).
\]
It follows that
\[
 t\circ A_x=t+\lambda_t(A_x)\subseteq t+F_t.
\]
Since the multiplicative cosets $t\circ A_x$ cover $B$, the additive
group admits the finite coset covering
\begin{equation}\label{eq:cover}
  (B,+)=\bigcup_{t\in T}(t+F_t).
\end{equation}

Moreover,
\[
  F_t=\lambda_t(F_x),
\]
so all subgroups occurring in \eqref{eq:cover} have the same additive
index.  Lemma~\ref{lem:neumann-covering} and \eqref{eq:T-bound} therefore imply
\begin{equation}\label{eq:Fx-index}
  |B:F_x|_+\leq |T|\leq ac\ell.
\end{equation}

Finally, let $\mathcal C=[x]_\circ$ and define
\[
  M=\bigcap_{y\in\mathcal C}\Fix(\lambda_y).
\]
Every factor has the same finite additive index as $F_x$, because if
$y=b\circ x\circ\bar b$, then
\[
 \Fix(\lambda_y)=\lambda_b(F_x).
\]
Hence $M$ has finite additive index and, using
\eqref{eq:Fx-index},
\[
 |B:M|_+
 \leq\prod_{y\in\mathcal C}|B:\Fix(\lambda_y)|_+
 \leq(ac\ell)^c.
\]

It remains to observe that $M$ is a left ideal.  For every $b\in B$,
\[
 \lambda_b\bigl(\Fix(\lambda_y)\bigr)
 =\Fix(\lambda_{b\circ y\circ\bar b}).
\]
Multiplicative conjugation by $b$ permutes $\mathcal C$, and thus
$\lambda_b(M)=M$.  Therefore $M$ is an additive subgroup invariant
under all $\lambda_b$, hence a left ideal.  Since $x\in\mathcal C$, we
also have $M\subseteq F_x$.

Finally, the additive core
$L=\Core_{(B,+)}(M)$ has finite index and lies in $M$.  Since
\[
 \lambda_d(b+M-b)=\lambda_d(b)+M-\lambda_d(b)
\]
for $b,d\in B$, every $\lambda_d$ preserves $L$.  Consequently $L$ is
a strong left ideal and $L\subseteq M\subseteq F_x$, as required.
\end{proof}

\begin{theorem}\label{thm:s-element}
Let $B$ be a skew brace and let $x\in B$.  If
\[
 x\in\FC(B,+)\cap\FC(B,\circ)\cap\lambda_f(B),
\]
then $x$ is an $(s)$-element.  Equivalently, Question~5.21 of
\cite{CPV} has an affirmative answer for arbitrary skew braces.
\end{theorem}

\begin{proof}
Let $M$ be the left ideal supplied by
Proposition~\ref{prop:fixed-left-ideal}.  Since $M$ is a left ideal,
for every $t\in B$,
\[
  t\circ M=t+\lambda_t(M)=t+M.
\]
In particular, its additive and multiplicative cosets coincide, so a
finite transversal $S$ exists for both group structures.

Take $g\in B$ and write $g=t\circ m$ with $t\in S$ and $m\in M$.
Put $v=\lambda_t(m)\in M$.  Then $g=t+v$, and
$\lambda_x(v)=v$ because $M\subseteq\Fix(\lambda_x)$.  Remembering
that $-(t+v)=-v-t$ in the possibly nonabelian additive group, we get
\begin{align*}
 x*g
 &=\lambda_x(g)-g\\
 &=\lambda_x(t+v)-(t+v)\\
 &=\lambda_x(t)+v-v-t\\
 &=\lambda_x(t)-t
 =x*t.
\end{align*}
Thus $g\mapsto x*g$ is constant on every coset $t\circ M$, and
\[
  |\{x*g:g\in B\}|\leq |B:M|_+<\infty.
\]

The remaining three sets in the definition of an $(s)$-element are
finite directly from the hypotheses:
\[
 \{g*x:g\in B\}=\{\lambda_g(x)-x:g\in B\},
\]
while the additive and multiplicative conjugates of $x$ form
$[x]_+$ and $[x]_\circ$, respectively.
\end{proof}

\begin{corollary}\label{cor:orbit-bound}
With the notation of Proposition~\ref{prop:fixed-left-ideal},
\[
  |\{x*g:g\in B\}|\leq(ac\ell)^c.
\]
Consequently, the union of the four families occurring in the
definition of an $(s)$-element has cardinality at most
\[
  a+c+\ell+(ac\ell)^c.
\]
\end{corollary}

\begin{corollary}\label{cor:s-characterization}
For every skew brace $B$ and every $x\in B$, the following are
equivalent:
\begin{enumerate}[label=\textup{(\arabic*)}]
 \item $x$ is an $(s)$-element;
 \item $x\in\FC(B,+)\cap\FC(B,\circ)\cap\lambda_f(B)$.
\end{enumerate}
\end{corollary}

\begin{proof}
The implication $(1)\Rightarrow(2)$ is immediate from the definition,
and $(2)\Rightarrow(1)$ is Theorem~\ref{thm:s-element}.
\end{proof}

\begin{remark}
Theorem~\ref{thm:s-element} is elementwise: the ambient skew brace need
not be a $\theta_f$-skew brace.  Thus it removes the global hypothesis
used in \cite[Proposition~5.17]{CPV} and includes the two-sided and
socle cases discussed there.
\end{remark}

\subsection{Consequences for finite presentability}

\label{sec:finite-presentation}

Presentations of skew braces were introduced and developed in
\cite{Trombetti}.  The recent work \cite{Free} provides workable
constructions of free skew braces in several relevant varieties, thereby
making the free objects underlying such presentations substantially more
explicit. A complementary framework is developed in \cite{poly}: almost polycyclic skew braces are finitely presented \cite[Remark 3.7]{poly}, and their residual finiteness and control by finite homomorphic images lead
to solvability of the word and conjugacy problems for structure skew braces of finite solutions.
The consequence below instead derives finite presentability from finite $\theta$-orbits. The statements in this section are not independent
finite-presentation theorems: they follow by combining
\cite[Proposition~5.8, Corollary~5.11 and Lemma~5.13]{CPV} with
\cite[Theorem~3.2 and Corollary~3.4]{Trombetti}.  We record them because
they connect the finite-orbit conditions considered above with the
finite-presentation framework of \cite{Trombetti}.

\begin{proposition}\label{prop:finite-presentation-consequence}\label{thm:finite-presentation}
Let $B$ be a skew brace generated, as a skew brace, by a finite subset
of $\theta_f(B)$.  Then $B$ is a finitely generated
$\theta_f$-skew brace and $\operatorname{Soc}(B)$ is a finitely
generated trivial ideal of finite index in $B$.  Moreover, $B$ is
finitely presented as a skew brace, and both $(B,+)$ and $(B,\circ)$
are finitely presented groups.

If $x_1,\dots,x_t$ generate $(B,+)$, then
\begin{equation}\label{eq:socle-bound}
 |B:\operatorname{Soc}(B)|
 \leq \prod_{i=1}^t |[x_i]_\theta|.
\end{equation}
In particular, every finitely generated $\theta_f$-skew brace is
finitely presented and virtually trivial.
\end{proposition}

\begin{proof}
By \cite[Proposition~5.8]{CPV}, $\theta_f(B)$ is a sub-skew brace of
$B$.  Since it contains a generating set of $B$, we have
$B=\theta_f(B)$.  Hence \cite[Corollary~5.11]{CPV} implies that both
$(B,+)$ and $(B,\circ)$ are finitely generated.

Choose additive generators $x_1,\dots,x_t$.  Since each $x_i$ belongs
to $\theta_f(B)$, \cite[Lemma~5.13]{CPV} shows that
$S=\operatorname{Soc}(B)$ has finite index in $B$ and gives
\eqref{eq:socle-bound}.  The group $(S,+)$ is finitely generated,
because it has finite index in the finitely generated group $(B,+)$.
For $s,u\in S$ one has
\[
 s\circ u=s+\lambda_s(u)=s+u=u+s=u\circ s,
\]
so $S$ is a trivial brace and $(S,+)=(S,\circ)$ is a finitely
generated abelian group.  Thus $S$ is finitely presented as a skew
brace by \cite[Corollary~3.4]{Trombetti}.

The quotient $B/S$ is finite.  In particular, it is finitely presented
as a skew brace, $(B/S,+)$ is finitely presented, and $(B/S,\circ)$ is
finitely generated.  Therefore all the hypotheses of
\cite[Theorem~3.2]{Trombetti} are satisfied, and $B$ is finitely
presented as a skew brace.  Finally, both underlying groups are finite
extensions of the finitely generated abelian group $S$ and are
therefore finitely presented.
\end{proof}

\begin{corollary}\label{cor:theta-generators}
Let $B$ be generated as a skew brace by $x_1,\dots,x_d$.  If, for each
$i$, the additive conjugacy class $[x_i]_+$ and the $\lambda$-orbit
$[x_i]_\lambda$ are finite, then $B$ is finitely presented.  In
particular this holds if, for every $i$,
\[
  x_i\in\FC(B,+)\cap\FC(B,\circ)\cap\lambda_f(B).
\]
Under the latter hypothesis every $x_i$ is an $(s)$-element.
\end{corollary}

\begin{proof}
The first two finiteness conditions say precisely that every $x_i$
belongs to $\theta_f(B)$.  Apply
Proposition~\ref{prop:finite-presentation-consequence}.  Under the
additional multiplicative condition, Theorem~\ref{thm:s-element}
gives the last assertion.
\end{proof}

For skew braces satisfying property $(S)$ one can replace the socle by
the annihilator.

\begin{corollary}\label{cor:S-braces}
If $B$ is a finitely generated $(S)$-skew brace, then
$\operatorname{Ann}(B)$ is a finitely generated trivial ideal of finite
index.  In particular, $B$ and both underlying groups are finitely
presented.
\end{corollary}

\begin{proof}
By \cite[Proposition~5.17]{CPV}, $B$ is $\theta_f$ and $(B,\circ)$ is
an $FC$-group.  The latter group is finitely generated by
\cite[Corollary~5.11]{CPV}, and hence its centre has finite index.
Proposition~\ref{prop:finite-presentation-consequence} gives that
$\operatorname{Soc}(B)$ has finite index.  Therefore
\[
 \operatorname{Ann}(B)
 =\operatorname{Soc}(B)\cap Z(B,\circ)
\]
has finite index as well.  It is a trivial ideal and is finitely
generated because it has finite index in either finitely generated
underlying group.  The final assertion follows from
Proposition~\ref{prop:finite-presentation-consequence}.
\end{proof}

\begin{remark}
Proposition~\ref{prop:finite-presentation-consequence} is a synthesis
of the cited results of \cite{CPV} and \cite{Trombetti}.  Its role here
is to connect the finite-orbit hypotheses used in this paper with the
existing theory of presentations; the independent results of the
paper are those stated in the introduction.
\end{remark}

\section{Finite-index sub-skew braces and finite generation}
\label{sec:schreier}

We now consider the direct analogue of Schreier's finite-index lemma \cite{Schreier}.
There is an immediate positive result whenever finite generation as a
skew brace can be strengthened to finite generation of one of the
underlying groups.

\begin{proposition}\label{prop:group-schreier}
Let $A$ be a finite-index sub-skew brace of a skew brace $B$.  If either
$(B,+)$ or $(B,\circ)$ is finitely generated, then $A$ is finitely
generated as a skew brace.
\end{proposition}

\begin{proof}
Assume first that $(B,+)$ is finitely generated.  Since
$|B:A|_+<\infty$, the group $(A,+)$ is finitely generated by 
Schreier's lemma.  Any additive generating set for $A$
generates $A$ as a skew brace.  The multiplicative case is identical.
\end{proof}

\begin{corollary}\label{cor:lambda-f-schreier}
Let $B$ be a finitely generated $\lambda_f$-skew brace.  Then every
finite-index sub-skew brace of $B$ is finitely generated.
\end{corollary}

\begin{proof}
By \cite[Theorem~4.8]{CPV}, finite generation of a $\lambda_f$-skew
brace is equivalent to finite generation of either underlying group.
Apply Proposition~\ref{prop:group-schreier}.
\end{proof}

Without the $\lambda_f$ condition, finite generation as a skew brace is
too weak.  We use the free right-nilpotent skew braces described in
\cite{Free}.  The sub-skew brace of elements of even additive length was already introduced in \cite[Theorem~5.5]{Free}, where it was used to prove that the corresponding free skew brace is not co-Hopfian. We retain that sub-skew brace and study its finite-generation properties.

\smallskip

Let $\mathcal{RN}_n$ denote the category of skew braces of
right-nilpotency class at most $n$, and put
\[
 B_n=\FSB_{\mathcal{RN}_n,\{x\}},
\]
the free object on one generator.  Thus $B_n$ is one-generated as a
skew brace.

We first isolate the structural input from \cite{Free}.

\begin{lemma}\label{lem:free-basis}
For $n\geq2$, let $q_n:B_n\to B_{n-1}$ be the canonical morphism
sending $x$ to $x$, and put $P_{n-1}=(B_{n-1},\circ)$.  Then
$(B_n,+)$ is the free group with basis
\[
 \mathcal Y_n=\{y_p:p\in P_{n-1}\},
 \qquad y_p=\lambda_p(x),
\]
where the indices are identified through the canonical quotient.
Moreover,
\begin{equation}\label{eq:lambda-shift}
 \lambda_b(y_p)=y_{q_n(b)\circ p}
 \qquad(b\in B_n,\ p\in P_{n-1}).
\end{equation}
\end{lemma}

\begin{proof}
The quotient description
$B_n/B_n^{(n)}\cong B_{n-1}$, with
$B_n^{(n)}\subseteq\ker\lambda$, is
\cite[Corollary~4.16]{Free}.  The stated free basis is
\cite[Proposition~5.2]{Free}.  If $\widehat p$ is a lift of $p$, then
\[
 \lambda_b(y_p)=\lambda_b\lambda_{\widehat p}(x)
 =\lambda_{b\circ\widehat p}(x).
\]
Since the $\lambda$-action factors through $q_n$, the index of the last
basis element is $q_n(b)\circ p$, which proves
\eqref{eq:lambda-shift}.
\end{proof}

Because $\Triv(C_2)$ belongs to $\mathcal{RN}_n$, freeness gives a
unique skew-brace epimorphism
\begin{equation}\label{eq:parity-map}
 \varepsilon_n:B_n\longrightarrow\Triv(C_2),
 \qquad \varepsilon_n(x)=1.
\end{equation}
For every $p\in P_{n-1}$,
\[
 \varepsilon_n(y_p)
 =\varepsilon_n(\lambda_p(x))=1.
\]
Thus $\varepsilon_n$ is precisely parity of the additive word length
in the basis $\mathcal Y_n$.  Uniqueness also gives
\begin{equation}\label{eq:parity-compatible}
 \varepsilon_n=\varepsilon_{n-1}q_n.
\end{equation}
Set
\[
 T_n=\ker\varepsilon_n.
\]
The sub-skew brace of even additive length considered in \cite[Theorem 5.5]{Free} is precisely $T_n$. Indeed, the preceding calculation identifies it with the kernel of the parity homomorphism $\varepsilon_n$. Consequently, $T_n$ is an ideal of $B_n$ and 
$$B_n/T_n \simeq \operatorname{Triv}(C_2),$$
so $T_n$ has index $2$ in $B_n$.

The new point here is the following failure of finite generation as a skew brace.

\begin{theorem}\label{thm:schreier-fails}
For every $n\geq3$, the one-generated skew brace $B_n$ has an ideal
$T_n$ of index $2$ which is not finitely generated as a skew brace.
\end{theorem}

\begin{proof}
Fix $n\geq3$, abbreviate $q=q_n$, $T=T_n$, and write
\[
 P=(B_{n-1},\circ),
 \qquad H=\ker\varepsilon_{n-1}\leq P.
\]
By \eqref{eq:parity-compatible}, $q(T)=H$.  The group $P$ is a free
group of infinite rank by the construction in \cite[Sections~4--5]{Free}.
Its index-$2$ subgroup $H$ is therefore also a free group of infinite
rank.  In particular, there are infinitely many group characters
$\psi:H\to C_2$.

Choose $r\in P\setminus H$.  Every $p\in P$ has a unique expression
$p=h\circ r^\delta$, where $h\in H$ and $\delta\in\{0,1\}$.  Given a
character $\psi:H\to C_2$, define
\[
 \eta_\psi(h\circ r^\delta)=\psi(h).
\]
For $k\in H$ and $p\in P$ one then has
\begin{equation}\label{eq:eta-translation}
 \eta_\psi(k\circ p)=\psi(k)+\eta_\psi(p).
\end{equation}

Since $(B_n,+)$ is free on $\mathcal Y_n$, the rule
\[
 \widetilde\Phi_\psi(y_p)=\eta_\psi(p)
\]
extends uniquely to a group homomorphism
$\widetilde\Phi_\psi:(B_n,+)\to C_2$.  If $b\in T$, then
$q(b)\in H$, so \eqref{eq:lambda-shift} and
\eqref{eq:eta-translation} give, on every free generator $y_p$,
\[
 \widetilde\Phi_\psi(\lambda_b(y_p))
 =\widetilde\Phi_\psi(y_p)+\psi(q(b))\varepsilon_n(y_p).
\]
Both sides extend as additive homomorphisms.  Hence, for every
$w\in B_n$,
\begin{equation}\label{eq:Phi-correction}
 \widetilde\Phi_\psi(\lambda_b(w))
 =\widetilde\Phi_\psi(w)+\psi(q(b))\varepsilon_n(w).
\end{equation}
On $T$ the correction term vanishes.  Therefore the restriction
\[
 \Phi_\psi=\widetilde\Phi_\psi|_T:T\longrightarrow C_2
\]
is invariant under every $\lambda_b$ with $b\in T$.  For $a,t\in T$,
\[
 \Phi_\psi(a\circ t)
 =\Phi_\psi(a+\lambda_a(t))
 =\Phi_\psi(a)+\Phi_\psi(t).
\]
Thus $\Phi_\psi:T\to\Triv(C_2)$ is a skew-brace homomorphism.

The map $\psi\mapsto\Phi_\psi$ is injective.  Indeed, if $e$ denotes
the identity of $P$ and $h\in H$, then
\[
 y_h-y_e\in T
 \qquad\text{and}\qquad
 \Phi_\psi(y_h-y_e)=\psi(h).
\]
Consequently, $T$ admits infinitely many skew-brace homomorphisms to
$\Triv(C_2)$.  If $T$ were generated as a skew brace by $d$ elements,
such a homomorphism would be determined by the $d$ images of those
generators, so there could be at most $2^d$ of them.  This
contradiction proves that $T$ is not finitely generated.
\end{proof}

The theorem separates two notions of finite generation which might
otherwise be conflated.

\begin{corollary}\label{cor:ideal-generated}
For every $n\geq3$, the ideal $T_n$ is finitely generated as an ideal
of $B_n$, but it is not finitely generated as a skew brace.
\end{corollary}

\begin{proof}
The skew brace $B_n$ is one-generated and $B_n/T_n\cong\Triv(C_2)$ is
finite, hence finitely presented.  Trombetti's
\cite[Lemma~3.1]{Trombetti} implies that $T_n$ is finitely generated as
an ideal of $B_n$.  The second assertion is
Theorem~\ref{thm:schreier-fails}.
\end{proof}

\begin{remark}
Thus neither ``strong left ideal'' nor even ``ideal'' can replace the
group-theoretic subgroup hypothesis in a finite-index Schreier theorem
for generation as a subbrace.  The obstruction already occurs at
index $2$, in a one-generated right-nilpotent skew brace of class $3$.
This does not settle the broader Nielsen--Schreier problem asking
whether subbraces of a free skew brace are free.
\end{remark}

The parity family has a sharp change between classes $2$ and $3$. The explicit model used below is taken from \cite[Theorem 5.7]{Free}, whereas the finite generating set is established here. 

\begin{proposition}\label{prop:class-two-parity}
The parity ideal $T_2$ of $B_2$ is generated as a skew brace by the four
elements
\[
 x_1-x_0,\qquad x_0+x_0,\qquad x_0+x_1,\qquad x_0-x_2.
\]
\end{proposition}

\begin{proof}
We use the model of \cite[Theorem~5.7]{Free}.  The additive group of
$B_2$ is the free group $F(x_i\mid i\in\Z)$.  If
\[
 \nu:F(x_i\mid i\in\Z)\longrightarrow\Z,
 \qquad \nu(x_i)=1,
\]
and $\vartheta(x_i)=x_{i+1}$, then
\[
 \lambda_w=\vartheta^{\nu(w)}.
\]
The parity ideal is $T_2=\ker(\nu\bmod2)$.

Let $C$ be the sub-skew brace generated by
\[
 a=x_1-x_0,\quad b=x_0+x_0,\quad
 c=x_0+x_1,\quad d=x_0-x_2.
\]
All four elements have even parity, so $C\subseteq T_2$.  Since
$\lambda_b=\vartheta^2$, closure under the brace operations gives
\[
 \vartheta^{2k}(z)=\lambda_{b^{\circ k}}(z)\in C
 \qquad(z\in C,\ k\in\Z).
\]
It follows that $C$ contains, for every $k\in\Z$,
\begin{align*}
 a_k&=x_{2k+1}-x_{2k},&
 b_k&=x_{2k}+x_{2k},\\
 c_k&=x_{2k}+x_{2k+1},&
 d_k&=x_{2k}-x_{2k+2}.
\end{align*}

Put $u_i=x_i-x_0$.  Starting from $u_0=0$, the recurrences
\[
 u_{2k+2}=-d_k+u_{2k},
 \qquad
 u_{2k}=d_k+u_{2k+2}
\]
show in both directions that $u_{2k}\in C$ for every $k$.  Then
\[
 u_{2k+1}=a_k+u_{2k}\in C.
\]
If $v_i=x_0+x_i$, we also have
\[
 v_{2k}=-u_{2k}+b_k,
 \qquad
 v_{2k+1}=-u_{2k}+c_k,
\]
so every $v_i$ belongs to $C$.

A Reidemeister--Schreier calculation for the kernel of
$\nu\bmod2$, using the transversal $\{0,x_0\}$, gives the additive
Schreier basis
\[
 \{u_i:i\in\Z,\ i\neq0\}\ \cup\ \{v_i:i\in\Z\}
\]
of $T_2$.  Hence $(T_2,+)\subseteq C$, and therefore $T_2=C$.
\end{proof}

\begin{remark}
For completeness, $B_1$ is the trivial brace on the infinite cyclic
group and $T_1=2\Z$ is one-generated.  Thus, for these canonical parity
ideals, class $1$ is one-generated, class $2$ is finitely generated by
the explicit set above, and every class $n\geq3$ is not finitely
generated as a skew brace.
\end{remark}

\vspace{2cm}
\paragraph{Acknowledgements.}
The authors are members of the \emph{National Group for Algebraic and Geometric Structures, and their Applications} (GNSAGA--INdAM), and of the non-profit association \emph{Advances in Group Theory and Applications}.

\newpage
{
\sloppy
\noindent
Massimiliano Di Matteo

\noindent
Dipartimento di Matematica e Fisica

\noindent
Università degli Studi della Campania  ``Luigi Vanvitelli''

\noindent
viale Lincoln 5, Caserta (Italy)

\noindent
e-mail: massimiliano.dimatteo@unicampania.it 
}

\bigskip\bigskip\bigskip
\bigskip

{
\sloppy
\noindent
Maria Ferrara

\noindent
Dipartimento di Ingegneria

\noindent
Facoltà di Ingegneria e Informatica

\noindent
Università Pegaso\\
Centro Direzionale Isola F2 - Napoli (Italy)

\noindent
e-mail: maria.ferrara1@unipegaso.it 

}

\end{document}